\documentclass[preprint,11pt]{article}
\usepackage[english]{babel}
\usepackage[latin1]{inputenc}
\usepackage{amsfonts}
\usepackage{amsthm}
\usepackage{amsmath}
\usepackage{latexsym}
\usepackage{amssymb}
\usepackage{mathrsfs}
\usepackage[T1]{fontenc}
\usepackage{scalerel,amssymb}

\usepackage{graphicx}

\newtheorem{fed}{Definition}[section]
\newtheorem{teo}[fed]{Theorem}
\newtheorem{cor}[fed]{Corollary}
\newtheorem*{teo*}{Theorem}
\newtheorem*{proof*}{Proof of Proposition 4.2}

\newtheorem{lem}[fed]{Lemma}
\newtheorem{prop}[fed]{Proposition}
\newtheorem{defi}[fed]{Definition}
\newtheorem{rem}[fed]{Remark}

\newcommand*{\Nwarrow}{\rotatebox[origin=c]{135}{\(\Longleftarrow\)}}

\newcommand*{\Swarrow}{\rotatebox[origin=c]{225}{\(\Longleftarrow\)}}

\def\bdem{\begin{proof}}
\def\edem{\renewcommand{\qed}{\hfill $\blacksquare$}
\end{proof}}

\def\cA{\mathcal{A}}
\def\cB{\mathcal{B}}
\def\cH{\mathcal{H}}
\def\cK{\mathcal{K}}

\def\QQ{[\cQ\cQ]_T}
\def\C{\mathbb{C}}
\def\cW{\mathcal{W}}
\def\ese{\mathcal{S}}
\def\ete{\mathcal{T}}

\def\cD{\mathcal{D}}

\def\cP{\mathcal{P}}
\def\cQ{\mathcal{Q}}
\newcommand{\pint}[1]{\displaystyle \left \langle #1 \right\rangle}

\begin{document}

\title{On partial orders of operators}

\author{M. Laura Arias$^{1,2}$ $^{a}$ and Alejandra Maestripieri$^{1,2}$ $^{b}$  \footnote{M. Laura Arias was partially supported by FONCYT (PICT 2017-
0883) and UBACyT (20020190100330BA). Alejandra Maestripieri was partially supported by CONICET PIP 0168}}
\date{}
\maketitle

\


{\fontsize {10}{10} \selectfont {{\sl {Keywords:}} {operator orders, generalized inverses, projections, Schur's complement}}}

\

\begin{abstract} 
	Characterizations of the star, minus and diamond orders of operators are given in various contexts and the relationship between these orders is made more transparent. Moreover, we introduce a new partial order of operators which provides a unified scenario for studying the other three orders. 
\end{abstract}

\section{Introduction}

This article presents a study of partial orders on the set $L(\cH,\cK)$ of bounded linear operators between Hilbert spaces $\cH$ and $\cK$. It focuses on three partial orders: the star order ($\overset{*}{\leq}$), the minus order ($\overset{-}{\leq}$) and the diamond order ($\overset{\diamond}{\leq}$). All these orders emerged as generalizations, in different senses, of the well-known L\"owner partial order. Recall that, given $A,B\in L(\cH),$ $A\leq B$ with respect to the L\"owner partial order if and only if $B-A$ is a positive semi-definite operator. The star order, defined by Drazin in \cite{MR486234}, then coincides with the L$\ddot{\text o}$wner order when restricted to the set of orthogonal projections. It emerged as a natural generalization of this fact for $L(\cH).$  On the other side, the minus order was introduced independently by Hartwig \cite{MR571255} and  by Nambooripad \cite{MR620922}. It replaces the adjoints operations  involved in the star order by inner inverses. Finally, another natural generalization of the star order is the diamond partial order, introduced on $\C^{n\times n}$ by Baksalary and Hauke \cite{MR1048800} when studying a new version of Cochran's theorem.  The diamond order was later extended in a natural way to operators acting on Hilbert spaces.
  
Several characterizations of the previous orders are scattered in the literature. However, a unified overview of this topic affording a clearer understanding of the behaviour of these orders and their relationship is missing. Our purpose here is to, in a clear and orderly way, characterize these orders in the contexts of matrix representations, inner inverses, operators ranges and operator equations. Moreover, we introduce a new partial order for operators, the plus order ($\overset{+}{\leq}$), which connects the three orders:

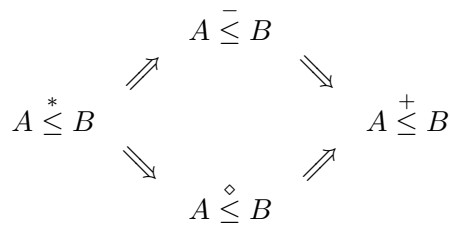
\begin{center}
$ \begin{matrix} & & A\overset{-}{\leq} B\\
& \Swarrow & &\Nwarrow \\
 A\overset{*}{\leq} B& & & & A\overset{+}{\leq} B  \\ 
&\Nwarrow & & \Swarrow\\
& & A\overset{\diamond}{\leq} B & 
\end{matrix}$
\end{center}

The plus order then provides a unified scenario for the study of the three orders. Characterizations of this order are given in terms of matrix representations and inner inverses.  Moreover, we observe that the previous orders can also be related to the notion of the bilateral shorted operator, a concept introduced in \cite{MR2214409} with the aim of extending to arbitrary operators in Hilbert spaces and pair of subspaces,  the Schur complement which was initially defined for positive operators and fixed subspace by Krein in \cite{MR0024575}. The bilateral shorted operator plays a role with operators related by the minus and diamond orders. This overview is concluded by studying the relation of the diamond and plus orders to certain factorizations of operators, including the polar decomposition and products of projections. If  $\cA,\cB$ are subsets of $L(\cH, \cK)$, $\cA\cB=\{AB: A\in\cA \; \text{and} \; B\in \cB\}$, then partial orderings on $\cA$ and $\cB$  induce a partial ordering on the whole set. With this approach we observe, among other results, that the  plus order on products of projections is equivalent to the minus order on the factors. 

The paper is ordered as follows. Section \ref{Preli} is devoted to recalling the definitions of the star, minus and diamond orders of operators, as well as some other concepts used throughout the paper. In Section \ref{matrix} we present an exhaustive description of these orders by means of matrix representations. Here, we explore the relationship between the minus and diamond orders with the bilateral shorted operator. Then, we introduce the plus order and describe it through matrix representations. Section \ref{inverses} is entirely devoted to studying the relationship between all of these orders and inner inverses, while Section \ref{rangos} studies their relationship with ranges and operator equations. Finally, in Section \ref{facto} we investigate the behaviour of the diamond and plus orders on certain factorizations involving projections and the polar decomposition.

\section{Preliminaries}\label{Preli}
In this article $\cH,\cK$ denote complex Hilbert spaces and $L(\cH, \cK)$ the set of bounded linear operators from $\cH$ to $\cK.$ When $\cH = \cK$ we write, for short,
$L(\cH)$. By $L(\cH)^+$ we  mean the set of positive semi-definite operators; i.e., $T\in L(\cH)^+$ if $\left\langle Tx,x\right\rangle\geq 0$ for all $x\in\cH.$ By $|T|$ we denote the modulus of $T$; i.e., $|T|=(T^*T)^{1/2}.$ The range of $T\in L(\cH,\cK)$ is written $R(T)$ and its nullspace $N(T)$. In addition,  $T^\dagger$ is the (possibly unbounded) Moore-Penrose inverse of $T\in L(\cH, \cK)$. The direct sum of two subspaces $\ese$ and $\ete$ of $\cH$ is expressed by $\ese\overset{.}{+}\ete$, and $\ese\oplus\ete$ means that the sum is orthogonal. We write $\cQ=\{Q\in L(\cH): Q^2=Q\}$ for the set of projections and $\cP=\{P\in L(\cH): P^2=P^*=P\}$ for the set of orthogonal projections. Given $T\in L(\cH,\cK)$, $P_T$ stands for the orthogonal projection onto $\overline{R(T)}.$ Throughout, we write $A\overset{s}{\leq} B$ whenever $R(A)\subseteq R(B)$ and $R(A^*)\subseteq R(B^*).$  The relationship $\overset{s}{\leq}$ defines a pre-order on $L(\cH, \cK)$, known as the \textit{space pre-order} on $L(\cH, \cK).$

Let us define the partial orders on $L(\cH,\cK)$ that we shall study in the article.

\begin{defi}
Let $A,B\in L(\cH, \cK)$. Write
\begin{enumerate}
\item $A$ $^*\hspace{-5pt}\leq B$ if $A^*A=A^*B$ and $R(A)\subseteq R(B)$; the left star order. A similar definition for the right star order ($\leq^*$). 
\item $A\overset{*}{\leq} B$ if $A^*A=A^*B$ and $AA^*=BA^*$; the star order.  
\item $A\overset{\diamond}{\leq} B$ if $A\overset{s}{\leq} B$ and $AA^*A=AB^*A$; the diamond order.
\item $A\overset{-}{\leq}B$  if there are projections $\tilde{Q}$ and $Q$ such that $A = \tilde{Q} B=BQ;$ the minus order. 
\end{enumerate}
\end{defi}
The following lemma collects some basic well-known properties of these orders. The proof is straightforward. 

\begin{lem} Let $A,B\in L(\cH, \cK)$. The following statements hold:

\begin{enumerate}
\item  $A^*\hspace{-5pt}\leq B$ if and only if $A=P_A B$ and $R(A)\subseteq R(B)$. Similarly, $A \leq^* B$ if and only if $A=BP_{A^*}$ and $R(A^*)\subseteq R(B^*)$.
\item  $A\overset{*}{\leq} B$ if and only $A=P_A B=BP_{A^*}.$
\item $AA^*A=AB^*A$ if and only if  and $P_A BP_{A^*}=A.$
\item If $A\overset{-}{\leq}B$ then the ranges of the projections $\tilde{Q}$ and $Q^*$ such that $A = \tilde{Q} B=BQ$ can be fixed so that $R(\tilde{Q})=\overline{R(A)}$ and $R(Q^*)=\overline{R(A^*)}.$
\end{enumerate}
\end{lem}

We also consider the notion of bilateral shorted operator which relies on the condition of weak complementability introduced in \cite{MR2214409}: 
Given $A\in L(\cH, \cK)$ and  $\ese$ and $\ete$ two closed subspaces of $\cH$ and $\cK$, respectively, consider the matrix representation of $A$  with respect to the decompositions $\cH=\ese\oplus\ese^\bot$ and $\cK=\ete\oplus\ete^\bot$:
\begin{equation}\label{AM}
A=\left(\begin{matrix} b & c \\ d & e\end{matrix}\right)
\end{equation}

\begin{defi}
An operator $A\in L(\cH, \cK)$ with the matrix decomposition as in (\ref{AM}) is called $(\ese,\ete)-$weakly complementable if $R(c^*)\subseteq R(|e|^{1/2})$ and $R(d)\subseteq R(|e^*|^{1/2}).$  In this case, the bilateral shorted operator of $A$ to the subspaces $\ese,\ete$ is defined as:
\begin{equation}
A_{/\ese,\ete}=\left(\begin{matrix} b-g^*f & 0 \\ 0 & 0\end{matrix}\right),
\end{equation}
 where $f=(|e^*|^{1/2}u)^\dagger d$,  $g=(|e|^{1/2})^\dagger c^*,$  and $e=|e^*|u$ is the polar decomposition of $e$.

\end{defi}

The weak complementability condition emerges as a generalization for arbitrary operators in Hilbert spaces of the  fact that if $A$ is a positive operator and $\ese=\ete$ then $R(d)\subseteq R(e^{1/2})$. If the stronger conditions $R(c^*)\subseteq R(e^*)$ and $R(d)\subseteq R(e)$ hold, then $A$ is  $(\ese,\ete)-$complementable. If $A$ is $(\ese,\ete)-$complementable then the matrix representation of $A$ with respect to the decompositions $\cH=\ese\oplus\ese^\bot$ and $\cK=\ete\oplus\ete^\bot$ is:
\begin{equation}\label{Amtxcomp}
A=\left(\begin{matrix} b & ye \\ ex & e\end{matrix}\right),
\end{equation}
for some $x\in L(\ese, \ese^\bot)$ and $y\in L(\ete^\bot, \ete).$ For details we refer the reader to \cite{MR2214409}.

\begin{lem}\label{shorted} Let $A\in L(\cH, \cK)$ and $\ese$ and $\ete$ be two closed subspaces  of $\cH$ and $\cK$, respectively. If $A$ is $(\ese,\ete)-$complementable and the matrix representation of $A$ with respect to the decompositions $\cH=\ese\oplus\ese^\bot$ and $\cK=\ete\oplus\ete^\bot$ is as in (\ref{Amtxcomp}), then
\begin{equation}
A_{/\ese,\ete}=\left(\begin{matrix} b-yex & 0 \\ 0 & 0\end{matrix}\right).
\end{equation}
\end{lem}
\begin{proof}
Assume that $A$ is $(\ese,\ete)-$complementable and let $A=\left(\begin{matrix} b & ye \\ ex & e\end{matrix}\right)$ be the matrix representation of $A$ with respect to the decompositions $\cH=\ese\oplus\ese^\bot$ and $\cK=\ete\oplus\ete^\bot$. Then:
  \begin{center}
 $c^*=e^*y^*=|e|u^*y^*=|e|^{1/2} |e|^{1/2}u^*y^*.$ 
 \end{center}
That is,  $g:=|e|^{1/2}u^*y^*=(|e|^{1/2})^\dagger c^*,$ and $d=ex=|e^*|ux=|e^*|^{1/2}|e^*|^{1/2}ux=|e^*|^{1/2}u|e|^{1/2}x.$ 

That is $f:=|e|^{1/2}x=(|e^*|^{1/2}u)^\dagger d$. Therefore, $b-g^*f=b-yu|e|^{1/2}|e|^{1/2}x=b-yex$ and so $A_{/\ese,\ete}=\left(\begin{matrix} b-yex & 0 \\ 0 & 0\end{matrix}\right).$

\end{proof}

The following concept also plays a role in what follows. 

\begin{defi} Let $B,C\in L(\cH)^+$.  The geometric mean of $B$ and $C$ is the operator:
\begin{equation}
B\# C:=\underset{\leq}{\max}\left\{X\in L(\cH)^+: \left(\begin{matrix} B & X \\ X & C\end{matrix}\right) \in L(\cH\oplus\cH)^+\right\}.\nonumber
\end{equation}
\end{defi}

For $B$ invertible and $T\in L(\cH)$, the algebraic Riccati equation in $X$ is given by:
\begin{equation}\label{Riccati}
 X^*  B^{-1} X -T X - X  T = C.
\end{equation} 
The next lemma gathers some useful properties of $B\# C$ . See  \cite{MR0482378} and \cite{MR2601957} for their proofs.
\begin{lem}\label{propmean} Let $B,C\in L(\cH)^+$. Then:
\begin{enumerate}
\item $B\# C=C\# B.$
\item If $B$ is invertible, then $B\# C=B^{1/2}(B^{-1/2}CB^{-1/2})^{1/2}B^{1/2}.$ 
\item If $B$ is invertible and $BT$ is selfadjoint, then the selfadjoint solution of the algebraic Riccati equation (\ref{Riccati}) is 
\begin{equation}
X = (T^*BT+ C)\# B +BT.\nonumber
\end{equation}
\end{enumerate}

\end{lem}

\section{Operator orders and matrix representations}\label{matrix}

Consider the following operator block matrices of $A$ and $B$ with respect to the decompositions $\cH=\overline{R(A^*)}\oplus N(A)$ and $\cK=\overline{R(A)}\oplus N(A^*):$ 
\begin{equation}\label{Amatrix}
A=\left(\begin{matrix} a & 0 \\ 0 & 0\end{matrix}\right) \; \textrm{and} \; B=\left(\begin{matrix} b_{11} & b_{12} \\ b_{21} & b_{22}\end{matrix}\right).
\end{equation}
In \cite{MR2737252} and \cite{MRD}, the matrix representation of $B$ is described when   $A\overset{*}{\leq} B$ or $A\overset{-}{\leq}B.$ However, the matrix representation of $B$ when $A\overset{\diamond}{\leq}B$ is missing. Here, we study this problem and we relate it with the notion of bilateral shorted operator.

For completeness of this overview, we include the results for the star and minus orders:

\begin{prop}\label{left*}\label{order**}
Let $A,B\in L(\cH,\cK)$. 
\begin{enumerate}
\item $A$ $^*\hspace{-5pt}\leq B$ if and only if $B=\left(\begin{matrix} a & 0 \\ b_{22}x & b_{22}\end{matrix}\right)$ for some $x\in L(\overline{R(A^*)}, N(A))$.
\item $A \leq^* B$ if and only  if there exists $y\in L(N(A^*), \overline{R(A)})$ such that $B=\left(\begin{matrix} a & yb_{22} \\ 0 & b_{22}\end{matrix}\right)$.
\item $A\overset{*}{\leq} B$ if and only if $B=\left(\begin{matrix} a & 0 \\ 0 & b_{22}\end{matrix}\right).$
\end{enumerate}
\end{prop}
\begin{proof} See \cite{MR2737252} and \cite{MRD}.

\end{proof}

\begin{prop} \label{minusmatrix}
Let $A,B\in L(\cH,\cK)$. The following conditions are equivalent:
\begin{enumerate}
\item $A\overset{-}{\leq}B;$ 
\item $B=\left(\begin{matrix} a+yb_{22}x & yb_{22} \\ b_{22}x & b_{22}\end{matrix}\right),$ for some $y\in L(N(A^*), \overline{R(A)})$ and $x\in L(\overline{R(A^*)}, N(A)).$
\end{enumerate}
\end{prop}
\begin{proof} See \cite{MRD}.

\end{proof}

\begin{prop}\label{diamond3}
Let $A,B\in  L(\cH, \cK)$. The following conditions are equivalent:
\begin{enumerate}
\item $A\overset{\diamond}{\leq}B.$
\item $b_{11}=a$, $R(b_{21})\subseteq R(B)$ and $R(b_{12}^*)\subseteq R(B^*).$
\end{enumerate}
\end{prop}
\begin{proof}
Assume that $b_{11}=a$. Hence, for any $x\in \overline{R(A^*)}$, $Bx=ax+b_{21}x,$ and so $R(A)\subseteq R(B)$ if and only if $R(b_{21})\subseteq R(B).$
In the same way, using that $B^*x=a^*x+b_{12}^*x$ for any $x\in \overline{R(A)},$ we get that $R(A^*)\subseteq R(B^*)$ if and only if $R(b_{12}^*)\subseteq R(B^*).$
\end{proof}

The matrix representations for the minus and diamond orders are connected to the notion of bilateral shorted operator.

\begin{teo}\label{minuscomplementable}
Let $A, B\in L(\cH, \cK)$. The following conditions are equivalent:
\begin{enumerate}   
\item $A\overset{-}{\leq} B;$ 
\item $B$ is $(\overline{R(A^*)}, \overline{R(A)})-$complementable and  
$$B_{/\overline{R(A^*)}, \overline{R(A)}}=A.$$
\end{enumerate}
\end{teo}
\begin{proof}
By Proposition \ref{minusmatrix},  $A\overset{-}{\leq} B$ if and only if there exist $x\in L(\overline{R(A^*)}, N(A))$  and $y\in L(N(A^*), \overline{R(A)})$ such that $B=\left(\begin{matrix} a+yb_{22}x & yb_{22} \\ b_{22}x & b_{22}\end{matrix}\right).$  Then, $B$ is $(\overline{R(A^*)}, \overline{R(A)})-$complementable, and Lemma \ref{shorted} leads to $B_{/\overline{R(A^*)}, \overline{R(A)}}=\left(\begin{matrix} a & 0 \\ 0 & 0\end{matrix}\right)$.

Conversely, if $B$ is $(\overline{R(A^*)}, \overline{R(A)})-$complementable then we can write $B=\left(\begin{matrix} b_{11} & yb_{22} \\ b_{22}x & b_{22}\end{matrix}\right)$ and, by Lemma \ref{shorted}, $B_{/\overline{R(A^*)}, \overline{R(A)}}=\left(\begin{matrix} b_{11}-yb_{22}x & 0 \\ 0 & 0\end{matrix}\right)$. But using that $B_{/\overline{R(A^*)}, \overline{R(A)}}=A,$ we get that $b_{11}-yb_{22}x=a$ or, $b_{11}=a+yb_{22}x.$ Then, $B=\left(\begin{matrix} a+yb_{22}x & yb_{22} \\ b_{22}x & b_{22}\end{matrix}\right)$ and, by Proposition \ref{minusmatrix}, $A\overset{-}{\leq} B$.
\end{proof}

 When $A\overset{\diamond}{\leq}B$, the next corollary gives a more explicit formula for $B$ than in Proposition \ref{diamond3} under an additional compatibility condition which, in particular, holds when $A$ has closed range.

\begin{prop}\label{diamcomp}
Let $A,B\in  L(\cH, \cK)$ be such that $B$ is $(N(A), N(A^*))-$complementable, and let $S:=B_{/N(A),N(A^*)}.$ The following conditions are equivalent:
\begin{enumerate}
\item $A\overset{\diamond}{\leq}B.$
\item $B=\left(\begin{matrix} a & xS \\ Sy & b_{22}\end{matrix}\right)$ for some $y\in L(\overline{R(A^*)}, N(A))$ and $x\in L(N(A^*), \overline{R(A)})$.  
\end{enumerate}
\end{prop}

\begin{proof} Since $B$ is $(N(A), N(A^*))-$complementable, by \cite[Corollary 4.9]{MR2214409}, $R(S)=R(B)\cap N(A^*).$ 
Assume that $A\overset{\diamond}{\leq}B.$ Then, $b_{11}=a.$ Furthermore, by Proposition \ref{diamond3}, $R(b_{21})\subseteq R(S).$ By Douglas' lemma there exists $y\in L(\overline{R(A^*)}, N(A))$ such that $b_{21}=Sy.$ In a similar way, $b_{12}^*=B^*_{/N(A^*),N(A)}x^*$ or, $b_{12}=xS$ for some $x\in L(\overline{R(A)}, N(A^*))$, where we used that $B^*_{/N(A^*),N(A)}=(B_{/N(A),N(A^*)})^*=S^*.$

The converse follows using again that $R(S)=R(B)\cap N(A^*)$ and Proposition \ref{diamond3}.
 
\end{proof}

For fixed operators $A,B$, the previous result provides a criterion  for $A\overset{\diamond}{\leq}B$ in terms of $S:=B_{/N(A),N(A^*)}.$ We next describe all the positive operators $B$ such that  $A\overset{\diamond}{\leq}B,$ for a fixed $A\in L(\cH)^+.$

\begin{teo}
Let $A,B\in  L(\cH)^+$, such that $A$ has closed range. Then, the following conditions are equivalent:
\begin{enumerate}
\item $A\overset{\diamond}{\leq}B;$ 
\item $B=\left(\begin{matrix} a & G^{-1}y^*b_{22} \\ b_{22}yG^{-1} & b_{22}\end{matrix}\right)$, 

for $y\in L(\overline{R(A)}, N(A))$, $b_{22}\in L(N(A))^+$ and $G=\frac{1}{2}+[(y^*b_{22}y+\frac{a}{4})\# a]a^{-1}.$
\end{enumerate}
\end{teo}

\begin{proof}

Assume that $A\overset{\diamond}{\leq}B$. Then, $b_{11}=a$ and the operator $S:=B_{/N(A),N(A)}$ is well-defined. In fact, in this case, $a$ is invertible and $B$ is trivially $(N(A), N(A))-$complementable.  Moreover, $S=b_{22}-b_{12}^*a^{-1} b_{12}$ and, by Proposition \ref{diamcomp}, $b_{12}^*=Sy$ for some $y\in L(\overline{R(A)}, N(A))$. Hence,  
\begin{equation}\label{ec1}
b_{22}y=b_{12}^*+b_{12}^*a^{-1} b_{12}y,
\end{equation}
 and multiplying both sides by $y^*$ we get
 $$y^*b_{22}y=y^*b_{12}^*+y^*b_{12}^*a^{-1} b_{12}y.$$
 Renaming $z:=b_{12}y$ and $c':=y^*b_{22}y$ and noting that $z=z^*$, the last equality turns into the Riccati equation: $$c'=z^*a^{-1} z+\frac{z+z^*}{2}.$$ Therefore, applying Lemma \ref{propmean} , we obtain that $b_{12}y=(\frac{a}{4}+y^*b_{22}y)\# a-\frac{a}{2}.$ Simple calculations show that $b_{12}=(\frac{1}{2}+[(y^*b_{22}y+\frac{a}{4})\# a]a^{-1})^{-1}y^*b_{22}.$

Conversely, suppose that item 2 holds. Then, from $G b_{12}=y^*b_{22},$ it can be easily deduced that 
\begin{equation}\label{6}
b_{12}^*+b_{12}^*a^{-1}((y^*b_{22}y+\frac{a}{4})\# a-\frac{a}{2})=b_{22}y,
\end{equation}
 and so $y^*b_{12}^*+y^*b_{12}^*a^{-1}((y^*b_{22}y+\frac{a}{4})\# a-\frac{a}{2})=y^*b_{22}y.$ If $D:=(y^*b_{22}y+\frac{a}{4})$,  we claim that $D\# a-\frac{a}{2}=b_{12}y.$ In fact,  
 \begin{eqnarray}
(\frac{1}{2}+(D\#a) a^{-1})b_{12}y&=&y^*b_{22}y=-\frac{a}{4}+D\nonumber\\
&=&-\frac{a}{4}+a^{1/2}(a^{-1/2}Da^{-1/2})^{1/2}a^{1/2}a^{-1}a^{1/2}(a^{-1/2}Da^{-1/2})^{1/2}a^{1/2}\nonumber\\
&=& -\frac{a}{4}+(D\#a) a^{-1} (D\#a)\nonumber\\
&=& (\frac{1}{2}+(D\#a) a^{-1})(D\#a -\frac{a}{2}),\nonumber
\end{eqnarray}
and so $D\# a-\frac{a}{2}=b_{12}y$ as desired. Hence, from (\ref{6}) we get that $b_{12}^*=b_{22}y-b_{12}^*a^{-1} b_{12}y$. Now, if  $X:=\left(\begin{matrix} 1+a^{-1} b_{12}y & 0 \\ -y & 0\end{matrix}\right),$ then $BX=A.$ Hence, $R(A)\subseteq R(B)$ and so $A\overset{\diamond}{\leq}B$.

\end{proof}

We end this section by introducing a new order relation in $L(\cH, \cK),$ called the {\it plus order}, which emerges as a natural generalization of both, the diamond order and the minus order. 

\begin{defi} Let $A,B\in L(\cH, \cK)$. Write $A\overset{+}{\leq}B$ if  $A\overset{s}{\leq}B$ and there are projections $\tilde{Q}$ and $Q$ such that $A = \tilde{Q}BQ.$ In such case, without loss of generality we can assume that $R(\tilde{Q})=\overline{R(A)}$ and $R(Q^*)=\overline{R(A^*)}.$
\end{defi}

\begin{lem}
The relationship $\overset{+}{\leq}$ defines a partial order on $L(\cH,\cK).$
\end{lem}
\begin{proof} Clearly, $\overset{+}{\leq}$ is a reflexive and antisymmetric relation. Let us prove that it is also transitive. If $A\overset{+}{\leq}B$ and $B\overset{+}{\leq}C$ then
 $R(A)\subseteq R(B)\subseteq R(C)$ and $R(A^*)\subseteq R(B^*)\subseteq R(C^*).$ Moreover, if $A = \tilde{Q}BQ$ and $B = \tilde{Q}_1CQ_1$ then $A = \tilde{Q}\tilde{Q}_1CQ_1Q$ and an easy computation shows that $\tilde{Q}\tilde{Q}_1$ and $Q_1Q$ are  projections. Hence, $A\overset{+}{\leq}C$.
\end{proof}

\begin{lem}\label{relation}  Let $A,B\in L(\cH,\cK).$ The next implications hold:
\begin{enumerate}
\item $A\overset{*}{\leq} B\Rightarrow A\overset{-}{\leq}B \Rightarrow A\overset{+}{\leq}B \Rightarrow A\overset{s}{\leq}B$.
\item $A\overset{*}{\leq} B\Rightarrow A\overset{\diamond}{\leq} B \Rightarrow A\overset{+}{\leq}B\Rightarrow A\overset{s}{\leq}B$.
\end{enumerate}
\end{lem}
\begin{proof}
Straightforward.
\end{proof}

\begin{teo}
Let $A,B\in L(\cH,\cK).$ The following conditions are equivalent:
\begin{enumerate}
\item $A\overset{+}{\leq}B$.
\item $B=\left(\begin{matrix} a+yb_{22}x &  yb_{22}\\ b_{22}x  & b_{22}\end{matrix}\right)+\left(\begin{matrix} yw &  0\\ w  & 0\end{matrix}\right)+\left(\begin{matrix} zx &  z\\ 0  & 0\end{matrix}\right),$

 where $y\in L(N(A^*), \overline{R(A)})$, $x\in L(\overline{R(A^*)}, N(A)),$ $w\in L(\overline{R(A^*)}, N(A))$, and $z\in L(N(A), \overline{R(A)})$ are such that $R\left(\left(\begin{matrix} yw &  0\\ w  & 0\end{matrix}\right)\right)\subseteq R(B)$ and $R\left(\left(\begin{matrix} x^*z^* &  0\\ x^*  & 0\end{matrix}\right)\right)\subseteq R(B^*).$
\item There exist two projections $\tilde{Q},Q$ such that  $A=\tilde{Q}BQ$, $R((I-\tilde{Q})BQ)\subseteq R(B)$ and $R((\tilde{Q}B(I-Q))^*)\subseteq R(B^*)$. 
\end{enumerate}
\end{teo}
\begin{proof} 
Let $\tilde{Q}=\left(\begin{matrix} 1 &  -y\\ 0  & 0\end{matrix}\right)$ and $Q=\left(\begin{matrix} 1 &  0\\ -x  & 0\end{matrix}\right)$ in the decompositions $\cK=\overline{R(A)}\oplus N(A^*)$ and $\cH=\overline{R(A^*)}\oplus N(A),$ respectively. Then,  $A=\tilde{Q}BQ$ if and only if 
\begin{eqnarray}
B&=&A+(I-\tilde{Q})B(I-Q)+(I-\tilde{Q})BQ+\tilde{Q}B(I-Q)\nonumber\\
&=& \left(\begin{matrix} a+yb_{22}x &  yb_{22}\\ b_{22}x  & b_{22}\end{matrix}\right)+\left(\begin{matrix} y(b_{21}-b_{22}x) &  0\\ b_{21}-b_{22}x  & 0\end{matrix}\right)+\left(\begin{matrix} (b_{12}-yb_{22})x & b_{12}-yb_{22} \\ 0  & 0\end{matrix}\right).\nonumber
\end{eqnarray}
In addition, $R\left(\left(\begin{matrix} y(b_{21}-b_{22}x) &  0\\ b_{21}-b_{22}x  & 0\end{matrix}\right)\right)=R((I-\tilde{Q})BQ))=R(BQ-\tilde{Q}BQ)=R(BQ-A)\subseteq R(B),$ if and only if $R(A)\subseteq R(B).$ Similarly, $R\left(\left(\begin{matrix} (b_{12}-yb_{22}x) & b_{12}-yb_{22}x \\ 0  & 0\end{matrix}\right)^*\right)\subseteq R(B^*)$ if and only if $R(A^*)\subseteq R(B^*).$ Therefore, $1\Leftrightarrow 2$.  

To prove $1\Leftrightarrow 3$, assume that there exist two projections $\tilde{Q},Q$ such that  $A=\tilde{Q}BQ$. Then $BQ=A+(I-\tilde{Q})BQ,$ and so $R((I-\tilde{Q})BQ)\subseteq R(B)$ if and only if $R(A)\subseteq R(B).$ In the same way, $B^*\tilde{Q}^*=A^*+(I-Q^*)B^*\tilde{Q}^*$, and so $R((\tilde{Q}B(I-Q))^*)\subseteq R(B^*)$ if and only if $R(A^*)\subseteq R(B^*).$

\end{proof}


\section{Operator orders and inner inverses}\label{inverses}

 Throughout, given $T\in L(\cH, \cK)$, $T[1]$ denotes the set of densely defined operators $T^-: \cD(T^-)\subseteq \cK\rightarrow \cH$  satisfying $R(T)\subseteq \cD(T^-)$ and $TT^-T = T$. The operator $T^-$ is called an inner inverse of $T$. The subset of $T[1]$ of those operators satisfying $T^-=T^-TT^-$ is denoted by $T[1,2].$ Observe that $T^\dagger$ is the unique element in $T[1,2]$ with $\cD(T^\dagger)=R(T)\oplus N(T^*),$  $N(T^\dagger)=N(T^*)$ and $R(T^\dagger)=\overline{R(T^*)}.$ If $T$ admits bounded inner inverses, then we say that $T$ is {\it relatively regular}. An operator $T$ is relatively regular if and only if $T$ has closed range. The aim of this section is to study operator orders by means of inner inverses. The first result we present in this direction follows from Proposition \ref{order**}:

\begin{cor}
	Let $A,B\in L(\cH, \cK)$ with closed ranges. The following conditions are equivalent:
	\begin{enumerate}
	\item $A\overset{*}{\leq} B$;
	\item $A^\dagger \overset{*}{\leq} B^\dagger$.
	\end{enumerate}
\end{cor}
\begin{proof}
	If $A,B$ have closed ranges and $A\overset{*}{\leq} B$ then, by Proposition \ref{order**},  $B=\left(\begin{matrix} a & 0 \\ 0 & b_{22}\end{matrix}\right)$. Hence $B^\dagger= \left(\begin{matrix} a^\dagger & 0 \\ 0 & b^\dagger_{22}\end{matrix}\right) $ and, again applying  Proposition \ref{order**}, we get that $A^\dagger \overset{*}{\leq} B^\dagger$. 
\end{proof}	

The next result can be found in \cite[Theorem 3.7]{MR3063879} for $A,B$ relatively regular bounded linear operators on Banach spaces.

\begin{prop}\label{minusinnerinverse} Let $A,B\in L(\cH, \cK).$ Then, $A\overset{-}{\leq}B$ if and only if $A\overset{s}{\leq}B$ and $B[1]\cap A[1]$ is not empty. 
\end{prop}

To prove Proposition \ref{minusinnerinverse}, we need the following lemma:

\begin{lem}\label{innerinv} Let $A,B\in L(\cH, \cK)$ such that $A\overset{s}{\leq}B$. If $B[1]\cap A[1]$ is not empty then $B[1]\subseteq A[1].$ 
\end{lem}
\begin{proof} If $A\overset{s}{\leq}B$ then, by Douglas's theorem, there exist $X\in L(\cK)$ and $Y\in L(\cH)$ such that $A=BX$ and $A^*=B^*Y^*.$ Furthermore, as $B[1]\cap A[1]$ is not empty then there exists $B^-\in B[1]$ such that $AB^-A=A.$  Consider $B^{=}\in B[1]$ then, since $R(A)\subseteq R(B)\subseteq \cD(B^{=}),$ $AB^{=}A=YB B^{=} BX=YB B^{-} BX=AB^{-}A=A,$ i.e., $B^{=}\in A[1].$
\end{proof}

\begin{proof*}
 Let $A\overset{-}{\leq}B$. Hence, there are  projections $\tilde{Q}$ and $Q$ such that $A = \tilde{Q} B=BQ.$ Thus, $A\overset{s}{\leq}B$. Moreover, let $B^-\in B[1]$, i.e., $B B^- B=B.$ Then, $A B^-A=\tilde{Q}B B^-B Q=\tilde{Q} B Q=A,$ and so $B^-\in A[1].$

Conversely, assume that  $A\overset{s}{\leq}B$ and $B[1]\cap A[1]$ is not empty. Define $Q:=B^\dagger A.$ then $Q\in L(\cH)$ since $R(A)\subseteq R(B).$ From Lemma \ref{innerinv}, it follows that $B^\dagger\in A[1]$ so that $AB^\dagger A=A.$ Therefore, $Q^2=B^\dagger AB^\dagger A=B^\dagger A=Q,$ and $Q$ is a projection. Also,
\begin{equation}\label{1}
BQ=BB^\dagger A=P_{\overline{R(B)}}|_{\cD(B^\dagger)}A=A,
\end{equation} 
because $R(A)\subseteq R(B).$

On the other hand, define $\tilde{Q}:=(B^*)^\dagger A^*\in L(\cK)$. Again by Lemma \ref{innerinv}, $\tilde{Q}^2=(B^*)^\dagger A^*(B^*)^\dagger A^*= (B^*)^\dagger A^*=\tilde{Q}.$ In this case,
 \begin{equation}\label{2}
B^* \tilde{Q}=B^*(B^*)^\dagger A^*=P_{\overline{R(B^*)}}|_{\cD((B^*)^\dagger)}A^*=A^*.
\end{equation} 
From (\ref{1}) and (\ref{2}), it follows that $A\overset{-}{\leq}B$. \hspace{250pt} $\square$

\end{proof*}

We recommend \cite[Proposition 3.8 and Corollary 3.19]{MR3682701}, for other characterizations of the minus order in terms of inner inverses.

We next study the equality $AA^*A =AB^*A$ involved in the diamond order in terms of inner inverses.
This result can be found for regular elements in rings in \cite[Lemma 5]{MR3175420}.

\begin{prop}\label{inndiamond}
Let $A,B\in L(\cH, \cK)$ . Then the following conditions are equivalent:
\begin{enumerate}
\item $AA^*A =AB^*A;$
\item $A^\dagger B A^\dagger=A^\dagger$ on $\cD(A^\dagger);$
\item $A^\dagger B A^\dagger \in A[1,2].$
\end{enumerate}
\end{prop}
\begin{proof}
$1\Rightarrow 2$. Let $AA^*A =AB^*A,$ or, equivalently, $A=P_A BP_{A^*}.$ Then, $BP_{A^*}x\in R(A)+R(A)^\bot$ for all $x\in\cH.$ Thus, $BP_{A^*}x\in \cD(A^\dagger)$ for all $x\in\cH,$ and so $AA^\dagger B A^\dagger A=A;$ i.e., $A^\dagger BA^\dagger\in A[1].$ Moreover, $A^\dagger BA^\dagger=A^\dagger AA^\dagger BA^\dagger AA^\dagger=A^\dagger AA^\dagger=A^\dagger$, as desired.

$2\Rightarrow 3$. Trivial.

$3\Rightarrow 1$. Suppose that $A^\dagger B A^\dagger \in A[1,2]$ . Then,    $A=AA^\dagger B A^\dagger A=P_{\overline{R(A)}}|_{\cD(A^\dagger)}BP_{A^*}=P_A BP_{A^*}$ .
\end{proof}

The next relationship between the minus order and diamond order comes from \cite{MR1048800}.  We include the proof for completeness.

\begin{cor}\label{diamondminus}
Let $A,B\in L(\cH, \cK)$ with closed ranges. The following conditions are equivalent:
\begin{enumerate}
\item $A\overset{\diamond}{\leq}B$;
\item $A^\dagger\overset{-}{\leq}B^\dagger$.
\end{enumerate}
\end{cor}
\begin{proof}
 Let $A\overset{\diamond}{\leq}B$ then $R(A)\subseteq R(B),$ $R(A^*)\subseteq R(B^*)$. Also, by Proposition \ref{inndiamond}, $A^\dagger=A^\dagger B A^\dagger=A^\dagger (B^\dagger)^\dagger A^\dagger;$ i.e., $B^\dagger\in A^\dagger[1].$ Therefore, by Proposition \ref{minusinnerinverse}, $A^\dagger\overset{-}{\leq}B^\dagger$. Conversely, if $A^\dagger\overset{-}{\leq}B^\dagger$ then $R(A)\subseteq R(B),$ $R(A^*)\subseteq R(B^*)$ and, by the proof of Proposition \ref{minusinnerinverse}, $B=(B^\dagger)^\dagger\in A^\dagger[1]$. Thus,  $A^\dagger B A^\dagger=A^\dagger$ or, equivalently, by Proposition \ref{inndiamond}, $AA^*A =AB^*A.$ Hence, $A\overset{\diamond}{\leq}B$.

\end{proof}

We turn to characterizing the equality $A = \tilde{Q}BQ$ for projections $Q,\tilde{Q}$ involved in the plus order in terms of inner inverses.

\begin{prop}
Let $A,B\in L(\cH, \cK)$ . Then the following conditions are equivalent:
\begin{enumerate}
\item There are two projections $\tilde{Q},Q$ such that $A = \tilde{Q}BQ;$
\item There exists $A^- \in A[1,2]$ such that $A^- B A^-=A^-$ on $\cD(A^-);$
\item There exists $A^- \in A[1,2]$ such that $A^- B A^- \in A[1,2]$.
\end{enumerate}
\end{prop}
\begin{proof}
$1\Rightarrow 2$. Let $\tilde{Q}$  and $Q$ be two  projections such that $A = \tilde{Q}BQ.$ Hence, $BQx\in R(A)+N(\tilde{Q})$ for all $x\in\cH.$ Thus, consider  $A^-\in A[1,2]$ with $\cD(A^-)=R(A)+N(\tilde{Q})$ and $AA^-=\tilde{Q}|_{\cD(A^-)}$, $A^-A=Q.$ Since $BQx\in \cD(A^-)$ for all $x\in\cH,$ we have that $AA^-BA^-A=A;$ i.e., $A^-BA^-\in A[1].$ Moreover, $A^-BA^-=A^-AA^-BA^-AA^-=A^-AA^-=A^-$ as desired.

$2\Rightarrow 3$. Trivial.

$3\Rightarrow 1$. Suppose that there exists $A^-\in A[1,2]$ such that $A^- B A^-\in A[1,2]$. Hence, $\cD(A^-)=R(A)+\ese$ for some closed subspace $\ese$ and $AA^-=Q_{R(A)//\ese}$ and $A^-A=Q_{\ete//N(A)}$ for some closed subspace $\ete.$ Therefore,  $A=AA^- B A^-A=Q_{R(A)//\ese}BQ_{\ete//N(A)}=Q_{\overline{R(A)}//\ese}BQ_{\ete//N(A)}$.
\end{proof}

\begin{rem}\label{adiamond}
Given $A,B\in L(\cH,\cK)$ such that $A \overset{+}{\leq} B$, we can change the inner products of $\cH,\cK$ in a convenient way so that $A\overset{\diamond}{\leq}B$ in the new spaces. In fact, if $A=\tilde{Q}BQ$, where $Q$ and $\tilde{Q}$ are projection in $\cH,\cK,$ respectively, the operator $A_Q:=Q^*Q+(I-Q^*)(I-Q)\in L(\cH)^+$ is invertible. Thus, the sesquilinear form $\pint{x,y}_{Q}:=\pint{A_Qx,y}$ defines an inner product equivalent to the inner product $\pint{ \; , \; }$. In the Hilbert space $\cH_{Q}:=(\cH, \pint{\; , \; }_Q),$ the projection $Q$ is orthogonal. Defining $A_{\tilde{Q}}$ in a similar way we get that $A\overset{\diamond}{\leq}B$ in $L(\cH_Q,\cK_{\tilde{Q}}).$ 
\end{rem}

\begin{prop}
Let $A,B\in L(\cH,\cK)$ with closed ranges. If $A \overset{+}{\leq} B$ then there exist $A^-\in A[1,2]$ and $B^-\in B[1,2]$ such that $A^- \overset{-}{\leq} B^-.$

\end{prop}
\begin{proof}
Suppose that $A \overset{+}{\leq} B$ and let $Q,\tilde{Q}$ be two projections such that $A=\tilde{Q}BQ.$ Then, by Remark \ref{adiamond}, $A,B\in L(\cH_Q, \cK_{\tilde{Q}})$ satisfy $A\overset{\diamond}{\leq}B$ and so, by Corollary \ref{diamondminus}, the Moore-Penrose inverses of $A$ and $B$ in  $L(\cK_{\tilde{Q}}, \cH_Q)$, denoted here by $A^-$ and $B^-$, respectively;  verify $A^- \overset{-}{\leq} B^-$. 
\end{proof}


\section{Operator orders, ranges and operator equations}\label{rangos}

\subsection{Operator orders and ranges}

Recall that given $A,B\in L(\cH, \cK)$, $A\overset{s}{\leq}B$ if and only if $R(A)\subseteq R(B)$ and $R(A^*)\subseteq R(B^*).$ This condition is in turn equivalent to 
$$R(B)=R(A)+R(B-A) \ \text{and} \ R(B^*)=R(A^*)+R(B^*-A^*).$$
It is easy to check that if $A\overset{*}{\leq}B$ then $A\overset{s}{\leq}B$ and the same holds replacing $\overset{*}{\leq}$ by the rest of the orders we study. In what follows, we gather several characterization of this additivity.

\begin{prop}\label{order*} Let $A,B\in L(\cH, \cK)$. The following conditions are equivalent:
\begin{enumerate}
\item  $A$ $^*\hspace{-5pt}\leq B;$ 
\item $R(B)=R(A)\oplus R(B-A)$.
\end{enumerate}
\end{prop}
\begin{proof}
 If $A$ $^*\hspace{-5pt}\leq B$ then, from Proposition \ref{left*}, there exists $x\in L(\overline{R(A^*)}, N(A))$ such that $B=\left(\begin{matrix} a & 0 \\ b_{22}x & b_{22}\end{matrix}\right)$. Thus, the result follows immediately. The converse is also trivial because $R(B-A)\subseteq N(P_A),$ and so $P_A B=A.$

\end{proof}

 As a corollary we retrieve the following from \cite{MR3682701} and \cite{MR2742334}:

\begin{cor}\label{I} Let $A,B\in L(\cH, \cK)$. The following conditions are equivalent: 
\begin{enumerate}
\item $A\overset{*}{\leq} B;$ 
\item $R(B)=R(A)\oplus R(B-A)$ and $R(B^*)=R(A^*)\oplus R(B^*-A^*);$
\item $R(B-A)=R(B)\ominus R(A)$ and  $R((B-A)^\dagger)=R(B^\dagger)\ominus R(A^\dagger).$ 
\end{enumerate}
\end{cor}	
\begin{proof}
The equivalence $1\Leftrightarrow 2$ follows from Proposition \ref{order*}.  On the other hand, $R(B^*)=R(A^*)\oplus R(B^*-A^*)$ is equivalent to $\overline{R(B^*)}=\overline{R(A^*)}\oplus \overline{ R(B^*-A^*)}$ or equivalently, $R(B^\dagger)=R(A^\dagger)\oplus R((B-A)^\dagger)$. Therefore $2\Leftrightarrow 3$ holds.
\end{proof}

A similar characterization can be given for the minus order. See \cite{MR3682701}.
 
\begin{prop}\label{ordenminus}
Let $A,B\in L(\cH,\cK)$ the following conditions are equivalent:
\begin{enumerate}
\item $A\overset{-}{\leq}B;$
\item $R(B)=R(A)\overset{.}{+}R(B-A)$ and $R(B^*)=R(A^*)\overset{.}{+}R(B^*-A^*);$
\item $\overline{R(B)}=\overline{R(A)}\overset{.}{+}\overline{R(B-A)}$ and $\overline{R(B^*)}=\overline{R(A^*)}\overset{.}{+}\overline{R(B^*-A^*)}.$
\end{enumerate}
\end{prop}

As a corollary  we get the next result concerning the diamond order:

\begin{cor}
Let $A,B\in L(\cH, \cK)$ with closed ranges. The following conditions are equivalent:
\begin{enumerate}
\item $A\overset{\diamond}{\leq}B$;
\item $R(B)=R(A)\overset{.}{+}R((B^\dagger-A^\dagger)^*)$ and $R(B^*)=R(A^*)\overset{.}{+}R(B^\dagger-A^\dagger).$
\end{enumerate}
\end{cor}
\begin{proof}
Apply Proposition \ref{ordenminus} and Corollary \ref{diamondminus}.
\end{proof}


\subsection{Operator orders and operator's equations}

Here, we  describe the operator orders by means of solutions of operator equations.

\begin{prop}
	Let $A,B\in L(\cH, \cK)$. The following conditions are equivalent:
		\begin{enumerate}
		\item $A$ $^*\hspace{-5pt}\leq B$;
		\item $B=A+C$ where $C\in L(\cH,\cK)$ is such that $R(C)\subseteq R(B)$ and $A^*C=0.$ 
		\end{enumerate}
\end{prop}
\begin{proof}
	If $A$ $^*\hspace{-5pt}\leq B$, $B=A+(B-A)$ and $A^*(B-A)=0$. Thus, $C:=B-A$ verifies $R(C)\subseteq R(B)$ and $A^*C=0.$  Conversely, if $B=A+C$ with $R(C)\subseteq R(B)$ and $A^*C=0$ then, trivially,   $A^*B=A^*A$ and $R(A)\subseteq R(B)$; i.e., $A$ $^*\hspace{-5pt}\leq B$.  
\end{proof}

In \cite{MR2293588},  the {\it logic order} between selfadjoint operators was introduced: for $A,B\in L(\cH)$ selfadjoint operators, $A$ is less than or equal to $B$ with respect to the logic order if there exists a selfadjoint operator $C\in L(\cH)$  such that $B=A+C$ and $AC=0.$ Notice that the above proposition shows that the logic order is the restriction of the left star order to selfadjoint operators. More generally, $A\overset{*}{\leq} B$ if and only if $B=A+C$ with  $C\in L(\cH,\cK)$ such that $A^*C=AC^*=0.$

\begin{cor} Let $A,B\in L(\cH, \cK)$ such that $A\overset{s}{\leq}B$. Then, $A\overset{*}{\leq}B$ if and only if $B^\dagger A$ and $(B^*)^\dagger A^*\in \cP$.
\end{cor}
\begin{proof}
If $A\overset{*}{\leq}B$ then $A=P_AB=BP_{A^*}.$ Thus, $B^\dagger A=B^\dagger B P_{A^*}= P_{B^*}P_{A^*}=P_{A^*}$ because $R(A^*)\subseteq R(B^*).$ Analogously, $(B^*)^{\dagger} A^*=(B^*)^{\dagger} B^* P_A=P_A.$ Conversely, if $B^\dagger A\in\cP$ then $B^\dagger AB^\dagger A=B^\dagger A$ and so $AB^\dagger A= A$. From this, $N(B^\dagger A)=N(A),$ so that $B^\dagger A=P_{A^*}.$ Therefore, $A=BB^\dagger A=BP_{A^*}.$ Similarly, from  $(B^*)^\dagger A^*\in \cP$, we get that $A^*=B^*P_{A}.$ Therefore, $A\overset{*}{\leq}B$.

\end{proof}

The following is an equivalent way of presenting the minus order. See \cite{MR840614}. 
\begin{prop}
	Let $A,B\in L(\cH,\cK)$. The following conditions are equivalent:
	\begin{enumerate}
	\item $A\overset{-}{\leq}B;$ 
	\item there exist $X\in L(\cK)$ and $Y\in L(\cH)$ such that $A=XB=BY$ and $A=XA.$
	\end{enumerate}
\end{prop}
\begin{proof}
	If $A\overset{-}{\leq}B$ then there exist projections $\tilde{Q}$ and $Q$ with $R(\tilde{Q})=\overline{R(A)}$ such that $A=\tilde{Q}B=BQ$ and $A=\tilde{Q}A$.
	
	Conversely, suppose that there exist $X\in L(\cK)$ and $Y\in L(\cH)$ such that $A=XB=BY$ and $A=XA.$ Define $\tilde{Q}:=XP_B$ and $Q:=YP_{A^*}.$ Then, $A=XB=XP_BB=\tilde{Q}B$ and $A=BY=BYP_{A^*}=BQ.$ Finish by proving that $\tilde{Q}$ and $Q$ are projections. Indeed,
	$(\tilde{Q}^*)^2=P_BX^*P_BX^*=(B^*)^\dagger B^*X^*P_BX^*=(B^*)^\dagger A^* P_BX^*=(B^*)^\dagger A^* X^*=(B^*)^\dagger A^* =(B^*)^\dagger B^*X^*=P_BX^*=\tilde{Q}^*.$
	Similarly, $Q^2=YP_{A^*}YP_{A^*}=YA^\dagger AYP_{A^*}=YA^\dagger XBYP_{A^*}=YA^\dagger XA=YA^\dagger A=YP_{A^*}=Q.$ 
	
\end{proof}

\begin{cor} Let $A,B\in L(\cH, \cK)$ such that $A\overset{s}{\leq}B$. Then, $A\overset{-}{\leq}B$ if and only if $B^\dagger A\in \cQ$.

\end{cor}
\begin{proof}
 If $A\overset{-}{\leq}B$ then, by the proof of Proposition \ref{minusinnerinverse},  we have that $Q:=B^\dagger A$ is a bounded projection. Conversely, if $B^\dagger A$ is a projection then $B^\dagger A=B^\dagger AB^\dagger A$ and so, as $R(A)\subseteq R(B)$,  $A=BB^\dagger A=BB^\dagger AB^\dagger A=AB^\dagger A$. Therefore, $B^\dagger\in A[1]$ and so, since $A\overset{s}{\leq}B$,  by Proposition \ref{minusinnerinverse}, $A\overset{-}{\leq}B$.

\end{proof}

\begin{prop}
Let $A,B\in L(\cH, \cK)$ . Then the following conditions are equivalent:
\begin{enumerate}
\item $AA^*A =AB^*A;$ 
\item $B=A+C$ with $C\in L(\cH,\cK)$ such that $AC^*A=0;$
\end{enumerate}
\end{prop}
\begin{proof}
 If item 1 holds, then $C=\left(\begin{matrix}0 & b_{12} \\ b_{21} & b_{22}\end{matrix}\right)$ verifies the desired conditions. Conversely, if item 2 holds then, clearly, $AA^*A =AB^*A.$

\end{proof}

\section{Partial orders onto projections}\label{facto}

Evidently, all the operator orders studied in this paper are closely linked to projections. Our goal in this section is to highlight this relationship in different senses. To this aim, we begin with the following lemma, where for fixed  order relation $\prec$ in $L(\cH)$, we denote by $[0, I]_\prec:=\{T\in L(\cH): 0 \prec T\prec I\}.$ 
\begin{lem}\label{I2}The following hold:
\begin{enumerate}
\item $[0, I]_{\overset{*}{\leq}}=[0, I]_{^*\leq}=[0, I]_{\leq^\ast}=\cP$.
\item $[0, I]_{\overset{-}{\leq}}=\cQ$. 
\item $[0, I]_{\overset{\diamond}{\leq}}=\cP\cP$.
\item $[0, I]_{\overset{+}{\leq}}=\cQ\cQ.$
\end{enumerate}
\end{lem}
\begin{proof}
It is straightforward.
\end{proof}

As we have said, the star order coincides with the L$\ddot{\text o}$wner order restricted to orthogonal projections. More precisely:

\begin{lem} \label{orP} On orthogonal projections the partial orders $\leq$, $\overset{*}{\leq}$, $\overset{-}{\leq}$, $\overset{\diamond}{\leq}$ and  $\overset{+}{\leq}$ coincide.
\end{lem}
\begin{proof}
Let $P,P'\in \cP.$ By Lemma \ref{relation}, it suffices to prove that $P\leq P'\Leftrightarrow P\overset{s}{\leq}P' \Leftrightarrow P\overset{*}{\leq}P'$. Clearly, $P\leq P'$ if and only if $R(P)\subseteq R(P'),$ i.e., $P\overset{s}{\leq}P'.$ Now, if $R(P)\subseteq R(P')$ then $P=P'P=PP'$ and so $P\overset{*}{\leq}P'.$ Conversely, if  $P\overset{*}{\leq}P'$ then $P=P'P=PP'$, i.e.,  $R(P)\subseteq R(P')$ or, equivalently, $P\overset{s}{\leq}P'.$ 
\end{proof}

The L$\ddot{\text o}$wner order can  also be related to the minus order on projections.

\begin{lem} Let $Q,Q'\in \cQ.$ The following conditions are equivalent:
\begin{enumerate}
\item $Q\overset{-}{\leq} Q';$
\item $Q\overset{s}{\leq}Q';$
\item $P_{Q}\leq P_{Q'}$ and $P_{Q'^*}\leq P_{Q^*}.$
\end{enumerate}
\end{lem}
\begin{proof}
If $Q\overset{s}{\leq}Q'$ then $R(Q)\subseteq R(Q')$ and $R(Q^*)\subseteq R((Q')^*)$. Thus, $Q=Q'Q=QQ'$ and so $Q\overset{-}{\leq} Q'.$ The converse always holds. Finally, $2\Leftrightarrow 3$ follows from the proof of Lemma \ref{orP}.
\end{proof}

Motivated by Lemma \ref{I2}, we focus now on studying in more detail the relationship between the diamond and plus orders with the sets $\cP\cP$ and  $\cQ\cQ$, respectively. Recall that if  $\cA,\cB$ are subsets of $L(\cH, \cK)$, $\cA\cB=\{AB: A\in\cA \; \text{and} \; B\in \cB\}$, then partial orderings on the domain of both components induce a partial ordering on the whole set. More precisely, given $(\cA, \leq_{\cA})$ and $(\cB, \leq_{\cB})$ two partially ordered subsets of $L(\cH, \cK)$, then if $A=A_1B_1$ and $B=A_2B_2$ with $A_1,A_2\in \cA$ and $B_1,B_2\in\cB$ then the relation $A\tilde{\leq}B$ if $A_1\leq_{\cA}A_2$ and $B_1\leq_{\cB}B_2$ defines a partial order on $\cA\cB.$
Under this scheme, we propose to study the diamond order onto the set $\cP\cP$ and the plus order onto $\cQ\cQ$. 
Before stating our next result, observe that if $T\in \cP\cP$ then $T=P_{T}P_{T^*}$, see \cite{MR2775769} for this fact and some other properties of $\cP\cP.$

\begin{prop}
Let $T,T'\in \cP\cP.$ The following conditions are equivalent:
\begin{enumerate}
\item $T\overset{\diamond}{\leq} T';$ 
\item $T\overset{s}{\leq} T'$.  

If $T'\in \cP\cP$ has closed range then the previous conditions are equivalent to:
\item $P_{T}\leq P_{T'}$ and $P_{T^*}\leq P_{T'^*}.$ 
\end{enumerate}
\end{prop}
\begin{proof}

$1\Rightarrow 2$. Obvious. 

$2\Rightarrow 1$. If $T\overset{s}{\leq} T'$ then $\overline{R(T)}\subseteq \overline{R(T')}$ and $\overline{R(T^*)}\subseteq \overline{R((T')^*)}$. So $TT^*T=P_T P_{T^*}P_{T^*}P_T P_T P_{T^*}=P_T P_{T^*} P_T P_{T^*}=T^2$ and $T(T')^*T=TP_{(T')^*}P_{T'}T=TT=T^2;$ i.e., $TT^*T=T(T')^*T$, and so $T\overset{\diamond}{\leq} T'$.

If $T'\in \cP\cP$ has closed range then $2\Leftrightarrow 3$  follows from the proof of Lemma \ref{orP}.

\end{proof}

Now, consider the set $\cQ\cQ.$ Denote by 	$\QQ:=\{(E,F)\in \cQ\times\cQ: T=EF, R(E)=\overline{R(T)} \; {\rm and} \; N(F)=N(T)\}.$  For characterizations and properties of the sets $\cQ\cQ$ and $\QQ,$ see \cite{MR3669130}. 

\begin{prop}
Let $T, T'\in \cQ\cQ$ and $T'$ with closed range. The following conditions are equivalent:
\begin{enumerate}
\item $T\overset{+}{\leq} T';$
\item for all $(E',F')\in [\cQ\cQ]_{T'}$ there exists $(E,F)\in\QQ$ such that $E \overset{-}{\leq}E'$ and $F\overset{-}{\leq}F'.$ 
\end{enumerate}
\end{prop}
\begin{proof}
$1\Rightarrow 2.$ If $T\overset{+}{\leq} T'$ then $\overline{R(T)}\subseteq R(T')$, and $\overline{R(T^*)}\subseteq R((T')^*)$ and there exist two projections $\tilde{Q},Q$ such that $T=\tilde{Q}T'Q.$ Without loss of generality we can assume that $R(\tilde{Q})=\overline{R(T)}$ and $R(Q^*)=\overline{R(T^*)}.$ Let $T'=E'F'=Q_{R(T')//\ese'}Q_{\cW'//N(T')}.$ Then, $\tilde{Q}E'$ is a projection. Indeed, $(\tilde{Q}E')^2=\tilde{Q}E'\tilde{Q}E'=\tilde{Q}E'$ since   $\overline{R(T)}\subseteq R(T')$. Moreover, $R(\tilde{Q}E')=\overline{R(T)}.$ Let $E:=\tilde{Q}E'.$ Analogously, $F:=F'Q$ is a projection with $R(F^*)=\overline{R(T^*)}.$ Therefore, $T=\tilde{Q}T'Q=\tilde{Q}E'F'Q=EF$. Furthermore, as $E=\tilde{Q}E'$, $R(E)\subseteq R(E')$ and $R(E')$ is closed, we get by \cite[Corollary 3.16]{MR3682701} that $E \overset{-}{\leq}E'$. Similarly, $F\overset{-}{\leq}F'.$

$2\Rightarrow 1.$ Let  $(E',F')\in [\cQ\cQ]_{T'}$ and $(E,F)\in\QQ$ such that $ E\overset{-}{\leq}E'$ and $F\overset{-}{\leq}F'.$ Then, from $E \overset{-}{\leq}E'$, there exist $\tilde{Q}_1,\tilde{Q}_2\in\cQ$ such that $E=\tilde{Q}_1E'=E'\tilde{Q}_2.$ Hence, $R(T)\subseteq \overline{R(T)}=R(E)\subseteq R(E')=R(T').$ Similarly, since $F\overset{-}{\leq}F'$, there exist $Q_1,Q_2\in\cQ$ such that $F=Q_1F'=F'Q_2$, and so $R(T^*)\subseteq\overline{R(T^*)}=R(F^*)\subseteq R((F')^*)=R((T')^*).$ Finally, $T=EF=\tilde{Q}_1E'F'Q_2=\tilde{Q}_1T'Q_2$, and so $T\overset{+}{\leq} T'$ as desired. 
\end{proof}

Finally, it is natural to  consider the relation between operator orders and polar decompositions. Consider polar decomposition of $T\in L(\cH, \cK)$, $T =|T^*|V_T$, where $V_T$ is a partial isometry with $N(V_T) = N(T)$.  In \cite{MR1670982} it is proven that $A\overset{*}{\leq}B$ if and only if $|A^*| \overset{*}{\leq}|B^*|$ and $V_A\overset{*}{\leq} V_B.$ Furthermore, if $|A^*| \overset{-}{\leq}|B^*|$ and $V_A\overset{-}{\leq} V_B$ then $A\overset{-}{\leq}B;$ and the converse is false in general; see also \cite{MR1670982}. Here, we complete this analysis by considering  $\overset{\diamond}{\leq}$ and $\overset{+}{\leq}$. We begin with the next lemma:

\begin{lem}\label{isop} The partial orders $\overset{*}{\leq}$, $\overset{-}{\leq}$ and $\overset{\diamond}{\leq}$ lead to the same relation within the class of partial isometries.
\end{lem}
\begin{proof}
	Let $F,G\in L(\cH, \cK)$ be partial isometries. Assume that $F\overset{\diamond}{\leq}G$. Then, $F=FF^*F=FG^*F$, $R(F)\subseteq R(G)$ and $R(F^*)\subseteq R(G^*).$ Therefore, $(FG^*)^2=FG^*$ and $FG^*(FG^*)^* FG^*=FG^*GF^*FG^*=FP_{G^*}F^*FG^*=FF^*FG^*=FG^*.$ Thus, $FG^*$ is a projection and a partial isometry; and so $FG^*$ is an orthogonal projection. Moreover, $R(FG^*)=R(FG^*GF^*)=R(FP_{G^*}F^*)=R(FF^*)$ and as a consequence $FG^*=FF^*.$ Analogously, from $F^*=F^*FF^*=F^*GF^*$ we get that $F^*G$ is an orthogonal projection with $R(F^*G)=R(F^*F)$, i.e., $F^*G=F^*F$. Summarizing, $GF^*=FF^*$ and $F^*G=F^*F$, and so $F\overset{*}{\leq}G$. The converse implication always holds.
	For the equivalence $F\overset{*}{\leq}G$ if and only if $F\overset{-}{\leq}G$ we referred to  \cite[Lemma 2]{MR1670982}.
\end{proof}

The partial order $\overset{+}{\leq}$ provides a new order relation within the partial isometries. Consider the partial isometries $F=\frac{1}{2}\left(\begin{matrix} \sqrt{2} & 0 \\ \sqrt{2} & 0\end{matrix}\right)$ and $G=\left(\begin{matrix} 0 & 1 \\ 1 & 0\end{matrix}\right),$ and the projections $Q=\left(\begin{matrix} 1/2 & 1/2 \\ 1/2 & 1/2\end{matrix}\right)$ and $\tilde{Q}=\left(\begin{matrix} 1 & 0 \\ \sqrt{2}-1 & 0\end{matrix}\right)$. Clearly, $F=QG\tilde{Q}$, $R(F)\subseteq R(G)$ and $R(F^*)\subseteq R(G^*),$ i.e., $F\overset{+}{\leq}G.$ However, $F\neq FG^*F,$ i.e., $F\overset{\diamond}{\nleq}G.$

\begin{prop}\label{poldiamond} Let $A=|A^*|V_A$ and $B=|B^*|V_B$ be the polar decompositions of $A,B\in L(\cH, \cK)$ respectively, and let $B$ have closed range. If $|A^*|\overset{\diamond}{\leq} |B^*|$ and $V_A\overset{\diamond}{\leq}V_B$ then $A\overset{\diamond}{\leq}B$.

\end{prop}
\begin{proof} 
If $V_A\overset{\diamond}{\leq}V_B$ then $R(A)\subseteq \overline{R(A)}=R(V_A)\subseteq R(V_B)=R(B)$ and $R(A^*)\subseteq\overline{R(A^*)}=R(V_A^*)\subseteq R(V_B^*)=R(B^*).$ Furthermore, by Lemma \ref{isop}, $V_A\overset{*}{\leq}V_B;$ i.e., $V_AV_A^*=V_AV_B^*$ and $V_A^*V_A=V_A^*V_B.$ 
On the other hand, if $|A^*|\overset{\diamond}{\leq} |B^*|$ then $|A^*||A^*||A^*|=|A^*||B^*||A^*|.$ Therefore,
$AA^*A=|A^*||A^*||A^*|V_A=|A^*||B^*||A^*|V_A=|A^*|V_AV_A^*|B^*||A^*|V_A=|A^*|V_AV_B^*|B^*||A^*|V_A=AB^*A.$ 

\end{proof}

\begin{rem} The converse of Proposition \ref{poldiamond} does not hold in general. For example, let $A=\left(\begin{matrix} 1 & 0 \\ 1 & 0\end{matrix}\right)$ and $B=\left(\begin{matrix} 0 & 2 \\ 2 & 0\end{matrix}\right)$. Then, $V_A=\frac{1}{2}\left(\begin{matrix} \sqrt{2} & 0 \\ \sqrt{2} & 0\end{matrix}\right)$ and $V_B=\left(\begin{matrix} 0 & 1 \\ 1 & 0\end{matrix}\right).$ A simple computation shows that $A\overset{\diamond}{\leq}B$ but $V_A\overset{\diamond}{\nleq}V_B.$

However, if $A\overset{\diamond}{\leq}B$ and  $V_A\overset{\diamond}{\leq}V_B$ then $|A^*|\overset{\diamond}{\leq} |B^*|$. Indeed, if $V_A\overset{\diamond}{\leq}V_B$ then by Lemma \ref{isop}, $V_AV_A^*=V_AV_B^*$. Now, as $A\overset{\diamond}{\leq}B$,  $|A^*|^3 V_A=AA^*A=AB^*A=|A^*|V_AV_B^*|B^*||A^*|V_A=|A^*|V_AV_A^*|B^*||A^*|V_A=|A^*||B^*||A^*|V_A.$ Postmultiplying by $V_A^*$, we get that $|A^*|^3=|A^*||B^*||A^*|;$ i.e, $|A^*|\overset{\diamond}{\leq} |B^*|$.
\end{rem}

\begin{prop}\label{pola}  Let $A=|A^*|V_A$ and $B=|B^*|V_B$ be the polar decompositions of $A,B\in L(\cH, \cK)$ respectively, and let $B$ with closed range. If $|A^*|\overset{*}{\leq} |B^*|$ and $V_A\overset{+}{\leq}V_B$ then $A\overset{+}{\leq}B$. 
\end{prop}
\begin{proof}
If $|A^*|\overset{*}{\leq} |B^*|$ then $|A^*|= |B^*|P_A.$ On the other hand, if $V_A\overset{+}{\leq}V_B$ then $R(A)\subseteq\overline{R(A)}=R(V_A)\subseteq R(V_B)=R(B)$ and  $R(A^*)\subseteq\overline{R(A^*)}=R(V_A^*)\subseteq R(V_B^*)=R(B^*)$ and $V_A=QV_B\tilde{Q}$ for some $Q, \tilde{Q}\in \cQ$ and $R(Q)=\overline{R(A)}.$ Hence, $A=|A^*|V_A=|B^*|P_AQV_B\tilde{Q}=|B^*|QV_B\tilde{Q}=|B^*|Q|B^*|^\dagger B\tilde{Q}.$ Now, $|B^*|Q|B^*|^\dagger\in \cQ$  and $\tilde{Q}\in\cQ.$ Thus, $A\overset{+}{\leq}B$. 
\end{proof}



\noindent $^1$ \fontsize {9}{9} \selectfont Instituto Argentino de Matem\'atica  ``Alberto P. Calder\'on''\ (CONICET), Buenos Aires, Argentina\\
\fontsize {9}{9} \selectfont $^2$ Dpto. de Matem\'atica, Facultad de Ingenier\'{i}a, Universidad de Buenos Aires, Buenos Aires, Argentina.\\

\fontsize {9}{9} \selectfont{$^a$  lauraarias@conicet.gov.ar, $^b$ amaestri@fi.uba.ar}

\end{document}